\documentclass[12pt]{article}
\usepackage{fullpage}
\usepackage{times}
\usepackage{amsmath}
\usepackage{latexsym}
\usepackage{amssymb,amsfonts,amsthm}
\usepackage{graphicx}
\usepackage{graphics}
\usepackage{float}
\usepackage{multirow}
\usepackage{multicol}
\usepackage{rotating}
\usepackage{setspace}
\usepackage{mathrsfs}
\usepackage{tikz}

 \def\p{{\cal P}}
 \def\Z{{\mathbb{Z}}}
 
 \newcommand{\GDD}{\ensuremath{\mbox{\sf GDD}}}

  \newcommand{\HWP}{\ensuremath{\mbox{\sf HWP}}}
    
 \renewcommand{\-}{\ensuremath{\text{--}}}

\let\oldmarginpar\marginpar
\renewcommand\marginpar[1]{\-\oldmarginpar[\raggedleft\footnotesize #1]%
{\raggedright\footnotesize #1}}

\newtheorem{theorem}{Theorem}
\newtheorem{lemma}[theorem]{Lemma}

\renewenvironment{proof}{\noindent{\it Proof}:}{\hfill $\blacksquare$}

\DeclareMathOperator{\lcm}{lcm}

\newenvironment{packed_enum}{
\begin{enumerate}
  \setlength{\itemsep}{1pt}
  \setlength{\parskip}{0pt}
  \setlength{\parsep}{0pt}
}{\end{enumerate}}

\newenvironment{packed_item}{
\begin{itemize}
  \setlength{\itemsep}{1pt}
  \setlength{\parskip}{0pt}
  \setlength{\parsep}{0pt}
}{\end{itemize}}

\begin{document}

\title{On the Hamilton-Waterloo Problem with triangle factors and $C_{3x}$-factors\footnote{This work is supported by the Scientific and Technical Research Council of Turkey (TUBITAK), under grant number 113F033.}}

\author{John Asplund\\
Dalton State College\\
Department of Technology and Mathematics\\
Dalton, GA 30720, USA\\
\\
David Kamin\\
University of Massachusetts Dartmouth\\
Kaput Center for Research and Innovation in STEM Education\\
Dartmouth, MA 02747, USA\\
\\
Melissa Keranen and Adri\'{a}n Pastine\\
Michigan Technological University\\
Department of Mathematical Sciences\\
Houghton, MI 49931, USA\\
\\
Sibel {\"O}zkan\\
Gebze Technical University\\
Department of Mathematics\\
Gebze, Kocaeli, Turkey 41400}

\maketitle

 \begin{abstract}
The Hamilton-Waterloo Problem (HWP) in the case of  $C_{m}$-factors and $C_{n}$-factors asks if $K_v$, where $v$ is odd (or $K_v-F$, where $F$ is a 1-factor and $v$ is even), can be decomposed into r copies of a 2-factor  made either entirely of $m$-cycles and $s$ copies of a 2-factor made entirely of $n$-cycles.  In this paper, we give some general constructions for such decompositions and apply them to the case where $m=3$ and $n=3x$. We settle the problem for odd $v$, except for a finite number of $x$ values.  When $v$ is even, we make significant progress on the problem, although open cases are left. In particular, the difficult case of $v$ even and $s=1$ is left open for many situations.
 \end{abstract}

\section{Introduction}
The Oberwolfach problem was first proposed by Ringel in 1967, and involves seating $v$ conference attendees at $t$ round tables over $\frac{v-1}{2}$ nights such that each attendee sits next to each other attendee exactly once. It is mathematically equivalent to decomposing $K_v$ into 2-factors where $K_v$ is the complete graph on $v$ vertices and each 2-factor is isomorphic to a given 2-factor $Q$. In the original statement of the problem, we have that $v$ must be odd.
It was later extended to the spouse-avoiding Oberwolfach problem, allowing for even $v$ by decomposing $K_v-F$, where $F$ is a 1-factor.

The Hamilton-Waterloo Problem (HWP) is an extension of the Oberwolfach Problem.  Instead of seating $v$ attendees at the same $t$ tables each night, the Hamilton-Waterloo problem asks how the $v$ attendees can be seated if they split their nights between two different venues.  The attendees will all spend the same $r$ nights in Hamilton, which has round tables of size $m_1, m_2, \dots ,m_k$, and $s$ nights in Waterloo, which has round tables of size $n_1, n_2, \dots ,n_p$ where $\sum_{i=1}^k m_i = \sum_{i=1}^p n_i =v$.  The case when $m_1 = m_2 =  \dots  =m_k =m$ and $n_1 = n_2 = \dots = n_p = n$ is called the Hamilton-Waterloo Problem with uniform cycle sizes, and this variant of the problem gets most of the attention. Graph theoretically, this problem is equivalent to decomposing $K_v$ (or $K_v-F$ when $v$ is even) into 2-factors where each 2-factor consists entirely of $m$-cycles (a $C_{m}$-factor) or entirely of $n$-cycles (a $C_{n}$-factor). Throughout this paper, the word \emph{factor} is assumed to be a 2-factor unless otherwise stated.  We frequently refer to a $C_{3}$-factor as a
\emph{triangle factor} and a Hamilton cycle as a \emph{Hamilton factor}. 

A \emph{decomposition} of a graph $G$ is a partition of the edge set of $G$. A decomposition of $K_{v}$ into $C_{m}$-factors is called a $C_{m}$-factorization.
We will refer to a solution to the Hamilton-Waterloo Problem with $r$ factors of $m$-cycles, $s$ factors of $n$-cycles, and $v$ points as a resolvable $(C_{m},C_{n})$-decomposition of $K_{v}$ into $r$
$C_{m}$-factors and $s$ $C_{n}$-factors, and we will let $(m,n)\-\HWP(v;r,s)$ denote such 
a decomposition.
In order for an  $(m,n)\-\HWP(v;r,s)$ to exist, it is clear that $r+s=\frac{v-1}{2}$ (or $r+s=\frac{v-2}{2}$, for even $v$), and both $m$ and $n$ must divide $v$.  These conditions are summarized in the following theorem.  

\begin{theorem}
\label{Ncond}
{\normalfont\cite{ABBE}} The necessary conditions for the existence of an $(m,n)\-\HWP(v;r,s)$ are
 \begin{packed_enum}
  \item If $v$ is odd, $r+s=\frac{v-1}{2}$,
  \item If $v$ is even, $r+s=\frac{v-2}{2}$,
  \item If $r>0$, $m|v$,
  \item If $s>0$, $n|v$.
 \end{packed_enum}
\end{theorem}

Recall that the Oberwolfach Problem involves seating $v$ conference attendees at $t$ round tables such that each attendee sits next to each other attendee exactly once.  The Oberwolfach Problem 
for constant cycle lengths was solved in {\normalfont\cite{AH,ASSW,HS}}. This is equivalent to the Hamilton-Waterloo Problem with $r=0$ or $s=0$.

\begin{theorem}
\label{OP}
{\normalfont\cite{AH,ASSW,HS}}
There exists a resolvable $m$-cycle decomposition of $K_{v}$ (or $K_{v}-F$ when $v$ is even) if and
only if $v \equiv 0 \pmod{m}$, $(v,m) \not = (6,3)$ and $(v,m) \not = (12,3)$.
\end{theorem}

An \emph{equipartite graph} is a graph whose vertex set can be partitioned into $u$ subsets of size $h$ such that no two vertices from the same subset are connected by an edge.  The complete equipartite graph with $u$ subsets of size $h$ is denoted $K_{(h:u)}$, and it contains every edge between vertices of different subsets.
Another key result solves the Oberwolfach Problem for constant cycle lengths over complete equipartite graphs (as opposed to $K_v$).  That is to say, with finitely many exceptions, $K_{(h:u)}$ has a resolvable $C_m$-factorization.

\begin{theorem}\label{equip}
 {\normalfont\cite{L}}  For $m\geq 3$ and $u\geq 2$, $K_{(h:u)}$ has a resolvable $C_m$-factorization if and only if $hu$ is divisible by $m$, $h(u-1)$ is even, $m$ is even if $u=2$, and $(h,u,m) \not \in \{(2,3,3), (6,3,3),(2,6,3),\\(6,2,6)\}$.
\end{theorem}

Much of the attention to the HWP has been dedicated to the case of triangle factors and Hamilton factors.
The results for this case have been summarized in the following theorem.

\begin{theorem}
\label{Hammy}
{\normalfont\cite{DL,DL2,HNR,LS}}
There exists a $(3,v)\-\HWP(v;r,s)$ with
\begin{packed_item}
\item $2 \leq s \leq \frac{v-1}{2}$ and $v \equiv3 \pmod{6}$ except possibly when:
\[v \equiv 15 \pmod{18} \mbox{ and }2 \leq s \leq \frac{v-3}{6}\mbox{ or } s=\frac{v+3}{6}+1,\]
\item $s=1$ and $v \equiv 3 \pmod{6}$ except when $v=9$ and possibly when:
\[v \in \{93,111,123,129,141,153,159,177,183,201,207,213,237,249\}.\] 
\item $2 \leq s\leq(v-2)/2$ and $v \equiv 0 \pmod{6}$ except possibly when $(v,s) \in \{(36,2), (36,4)\}$ 
or when $v \equiv 12 \pmod{18}$ and $2 \leq s \leq (v/6)-1$; and
\item $s=1$ and $v \equiv 0 \pmod{6}$ except possibly when $v=18$, $v \equiv 12 \pmod{18}$ or
$v \equiv 6 \pmod{36}$.
\end{packed_item}
\end{theorem}

When considering the $\HWP$ for triangle factors and Hamilton factors, the focus is on a specific
case of the problem.  This paper considers a more general family of decompositions, namely, triangle factors and $3x$-factors of $K_{v}$ for any $v$ that is divisible by both $3$ and $3x$. In
this instance of the problem, $v$ is of the form $3xy$.
When $x=1$, the problem of finding a $(3,3x)\-\HWP(v;r,s)$
is simply that of finding a resolvable $C_{3}$-factorization of $K_{v}$, which is also known as
a Kirkman triple system ($KTS(v)$).  It was shown in 1971 by Ray-Chadhuri and Wilson {\normalfont\cite{RW}}
and independently by Lu (see {\normalfont\cite{Lu}}) that a $KTS(v)$ exists if and only if $v \equiv 3 \pmod{6}$.
When $y=1$, then the problem
asks for a decomposition of $K_{v}$ into triangle factors and Hamilton cycles. This case is
addressed in {\normalfont\cite{DL}}, {\normalfont\cite{DL2}}, and {\normalfont\cite{HNR}}, and the results were presented in Theorem~{\normalfont\ref{Hammy}}. Therefore, we focus on the cases where $x \geq 2$ and $y \geq 2$. It is
a different type of decomposition than what was considered in {\normalfont\cite{DL,DL2,HNR}}, because in our case, we let both $x$ and $y$ vary . However, as expected, the results given in Theorem~{\normalfont\ref{Hammy}} can be used in the decompositions we are interested in.

The Hamilton-Waterloo Problem was studied in 2002 by Adams, et. al. {\normalfont\cite{ABBE}}.  The paper provides solutions to all Hamilton-Waterloo decompositions on less than 18 vertices. 
Some notable results involving $v=6$ and $v=12$ will be relevant to this paper.

\begin{theorem}
{\normalfont\cite{ABBE}}
\label{6and12}
There exists a $(3,6)\-\HWP(12;r,s)$ if and only if $r+s=5$ except $(r,s)=(5,0)$.
There exists a $(3,12)\-\HWP(12;r,s)$ if and only if $r+s=5$ except $(r,s)=(5,0)$.
There exists a $(3,6)\-\HWP(6;r,s)$ if and only if $r+s=2$ except $(r,s)=(2,0)$.
\end{theorem}

The authors in {\normalfont\cite{ABBE}} also developed a tripartite construction that
could be used when considering
%
%
$m=3$ and $n=3x$. However, it leaves many open cases, because it relies on the existence of a $(3,v)\-\HWP(v;r,s)$ for all $(r,s)$ and for all
$v \equiv 3 \pmod{6}$.  According to Theorem~{\normalfont\ref{Hammy}}, there are some gaps in
the existence of these.  The problem is that the construction given in {\normalfont\cite{ABBE}} uses a uniform decomposition of $K_{(x:3)}$.
Therefore, we proceed in this paper by developing a new construction that is a bit more general, and in particular, depends on the decomposition of $K_{(x:3)}$ into $r_{p}$ $C_{m}$-factors and $s_{p}$ $C_{n}$-factors.  The flexibility in this construction allows us to settle all but 14 cases of the existence of a $(3,3x)\-\HWP(3xy;r,s)$ for all possible $(r,s)$ whenever both $x\geq 3$ and $y \geq 3$ are odd. We also introduce a modified construction that is used in the cases where at least one of $x$ or $y$ is even. We give almost complete results for these cases as well. In Section~{\normalfont\ref{When $x$ is small}} we handle the cases when $x \in \{2,4\}$ and collect all of the results into a summarizing theorem in Section~{\normalfont\ref{conclusions}}.

\section{Constructions}
\label{Constructions}

In this section, we develop constructions that will later be used to prove our main results about 
the Hamilton-Waterloo Problem in the case of triangle factors and $C_{3x}$-factors.

Recall that $K_{(x:3)}$ is the complete multipartite graph with $3$ parts of size $x$.
Let the parts be $G_0$, $G_1$ and $G_2$ and the vertices be $(a,b)$ with $0\leq a \leq 2, 0\leq b \leq x-1$. 
Consider the edge $\{(a_{1},b_{1}),(a_{2},b_{2})\}$ which has one vertex from $G_{a_{1}}$
and one vertex from $G_{a_{2}}$. With computations being done in $\Z_{x}$, we say this edge has difference $b_{2}-b_{1}$.
Let $T_x(i)$ for $0\leq i\leq x-1$ be the subgraph of $K_{(x:3)}$ obtained by taking all edges of difference: $2i$ between vertices of  $G_0$ and vertices of $G_1$, $-i$ between $G_1$ and $G_2$, and $-i$ between $G_2$ and $G_0$.
\begin{lemma}
$T_x(i)$ is a triangle factor of $K_{(x:3)}$ for any $i$.
\end{lemma}

\begin{proof}
It is easy to see that the triangles are of the form $\{(0,k),(1,k+2i),(2,k+i)\}$ for every $0 \leq k \leq x-1$.
\end{proof}

Let $H_x(i,j)$ be the subgraph of $K_{(x:3)}$ obtained by taking all edges of difference: $2i$ between $G_0$ and $G_1$, $-i$ between $G_1$ and $G_2$, and $-j$ between $G_2$ and $G_0$.
\begin{lemma}
If $gcd(x,i-j)=1$ then $H_x(i,j)$ is a Hamiltonian cycle of $K_{(x:3)}$.
\end{lemma}

\begin{proof}
Since the edges are given by differences it is clear that all vertices have degree 2.
We need to show that all the vertices are connected. We will first show that there is a path between any 2 vertices of $G_0$.
Without loss of generality, we will show that $(0,0)$ is connected to $(0,k)$ for any $k$.
Starting at $(0,0)$, we may traverse the path:  $(0,0),(1,2i),(2,i),(0,i-j)$. Thus the next time that we reach $G_0$ it is via the vertex $i-j$. Since $gcd(x,i-j)=1$, the order of $i-j$ in the cyclic group $\mathbb{Z}_x$ is $x$. Therefore, any $k$ modulo $x$ can be written as $k'(i-j)$, which means that we
reach the vertex $(0,k)$ after visiting the part $G_{0}$ $k'$ times. Hence $(0,0)$ is connected to all the vertices of $G_0$ via a path. 

Because we are taking every edge of a particular difference, it follows that every vertex in $G_1$ is connected to a vertex in $G_0$, and the same is true for vertices in $G_2$. Hence all the vertices are connected, and the cycle is Hamiltonian, as we wanted to prove.
\end{proof}

\subsection{When $x$ is Odd}
\label{Odd $x$}

We can think of a decomposition of a graph $G$ as a partition of the edge set or as a union of edge disjoint subgraphs. This means that a decomposition of $G$ can be given by $E(G)=\cup E(F_i)$ or by $G=\oplus F_i$, where each $F_i$ is an edge disjoint subgraph of $G$.
The next lemma shows that $K_{(x:3)}$ can be decomposed entirely into triangle factors or Hamilton cycles when $x$ is odd.

\begin{lemma}
\label{Alg1}
Let $x$ be an odd integer, and let $\phi$ be a bijection of the set $\{0,1,\ldots,x-1\}$ into itself. Then
\[
K_{(x:3)}=\bigoplus_{i=0}^{x-1} T_x(i)=\bigoplus_{i=0}^{x-1} H_x(i,\phi(i))
\]
\end{lemma}

\begin{proof}
To prove the first equality,
\[
K_{(x:3)}=\bigoplus_{i=0}^{x-1} T_x(i)
\]
we need to show that between each pair of parts in $K_{(x:3)}$, each difference is covered by the edges in one of the triangle factors exactly once. It is clear that edges of difference $k$ between $G_1$ and $G_2$  and between $G_2$ and $G_0$ are covered in $T_x(k)$. Now consider groups $G_{0}$ and $G_{1}$. Each factor $T_{x}(k)$ uses the difference $2k$. Because $gcd(x,2)=1$, the order of $2$ in the cyclic group $\mathbb{Z}_x$ is $x$. So it follows that any $k$ modulo $x$ can be written as $2k'$, and thus the difference $k$ between $G_0$ and $G_1$ is covered in $T_x(k')$. 
Notice that we cover the edges of exactly one difference between any two parts per subgraph, and we only have $x$ subgraphs. This together with the fact that we are covering all the differences imply that we cover each difference exactly once. Thus it is equivalent to decomposing $K_{(x:3)}$.

The second equality 
\[
\bigoplus_{i=0}^{x-1} T_x(i)=\bigoplus_{i=0}^{x-1} H_x(i,\phi(i))
\]
is true because we again cover each difference between any pair of parts exactly once by the edges in the factors.
\end{proof}

Notice that the subgraph $H_x(i,i)$ is the same as $T_x(i)$. Therefore, decomposing $K_{(x:3)}$ into $s$ Hamilton cycles and $x-s$ triangle factors is equivalent to finding a bijection $\phi$ such that $\gcd(x,i-\phi(i))=1$ for $s$ elements of $\{0,1,\ldots,x-1\}$ and $\phi(i)=i$ for the rest.

\begin{theorem}\label{phitheorem}
\label{K(x:3)}
Let $x$ be odd and let $s\in \{0,2,3,\ldots,x\}$. Then:
\begin{packed_item}
\item there exists a bijection $\phi$ on the set $\{0,1,\ldots,x-1\}$ with $\gcd(x,i-\phi(i))=1$ for $s$ elements and $r=x-s$ fixed points; and 
\item $K_{(x:3)}$ can be decomposed into $s$ Hamiltonian cycles and $r=x-s$ triangle factors.
\end{packed_item}
\end{theorem}

\begin{proof}
If $s=0$ we just use the identity mapping. Let $2\leq s \leq x$, and let $e$ be the smallest integer such that $s \leq 2^e +1$. We have 
\[
2^{e-1}+1<s\leq \min\{2^e+1,x\}=t.
\]
Let $r=t-s$ and define $\phi$ as follows:
\[
\phi(i)=\left\lbrace \begin{array}{lcl} 

0 & \text{for}  & i=1\\
i+2 & \text{for}  & i\equiv 0 \pmod 2, 0\leq i \leq s-3\\
i-2 & \text{for}  & i\equiv 1 \pmod 2, 3\leq i \leq s-1 \\
s-2 & \text{for}  & i\equiv 0 \pmod 2, i=s-1\\
s-1 & \text{for}  & i\equiv 0 \pmod 2, i=s-2\\
i &\text{for} & s\leq i \leq x-1\\

\end{array}\right.
\]

It is an easy exercise to check that $\phi$ is a bijection with $r=x-s$ fixed points. Furthermore, for any non-fixed point we have $(i-\phi(i))\in \{\pm1,\pm2\}$ and, because $x$ is odd, $\gcd(x,i-\phi(i))=1$. Hence by Lemma~{\normalfont\ref{Alg1}},
\[
K_{(x:3)}=\bigoplus_{i=0}^{x-1} H_x(i,\phi(i))
\]
is a decomposition of $K_{(x:3)}$ into $s$ Hamiltonian cycles and $r=x-s$ triangle factors.
\end{proof}

Unfortunately this construction only works when $x$ is odd. For the cases when $x$ is even we can get a similar result, although only when $x=2\bar{x}$, with $\bar{x}$ odd. 

\subsection{When $x$ is Even}
\label{Even $x$}

In this subsection, we develop a construction similar to what is described in Section~{\normalfont\ref{Odd $x$}}. It relies on the following decomposition of $K_{(4:3)}$ into triangle factors.
Define $\Gamma(i)$ for $i\in \{0,1,2,3\}$ as follows.

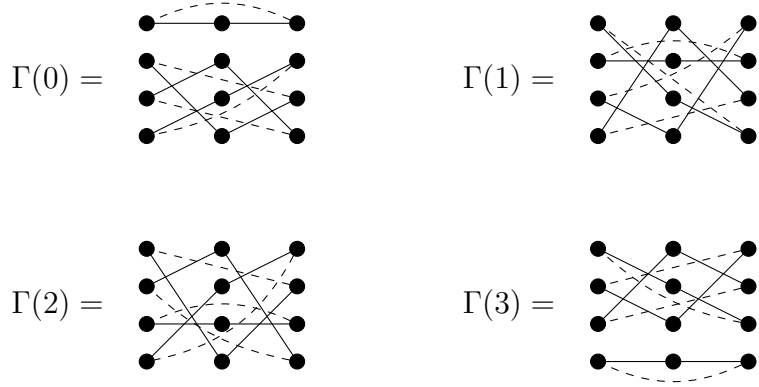
\begin{center}
\begin{tikzpicture}[every node/.style={draw,shape=circle,fill=black}]
\draw (0,0) node [scale=.5] (000){};
\draw (1,0) node [scale=.5] (010){};
\draw (2,0) node [scale=.5] (020){};
\draw (0,-.5) node [scale=.5](001){};
\draw (1,-.5) node [scale=.5](011){};
\draw (2,-.5) node [scale=.5](021){};
\draw (0,-1) node [scale=.5](002){};
\draw (1,-1) node [scale=.5](012){};
\draw (2,-1) node [scale=.5](022){};
\draw (0,-1.5) node [scale=.5](003){};
\draw (1,-1.5) node [scale=.5](013){};
\draw (2,-1.5) node [scale=.5](023){};

\draw (000) to (010);
\draw (010) to (020);
\draw (020) [bend right=25, dashed] to (000);
\draw (001) to (013);
\draw (013) to (022);
\draw (022) [dashed] to (001);
\draw (002) to (011);
\draw (011) to (023);
\draw (023) [dashed] to (002);
\draw (003) to (012);
\draw (012) to (021);
\draw (021) [bend left=15, dashed] to (003);

\draw (6,0) node [scale=.5] (100){};
\draw (7,0) node [scale=.5] (110){};
\draw (8,0) node [scale=.5] (120){};
\draw (6,-.5) node [scale=.5](101){};
\draw (7,-.5) node [scale=.5](111){};
\draw (8,-.5) node [scale=.5](121){};
\draw (6,-1) node [scale=.5](102){};
\draw (7,-1) node [scale=.5](112){};
\draw (8,-1) node [scale=.5](122){};
\draw (6,-1.5) node [scale=.5](103){};
\draw (7,-1.5) node [scale=.5](113){};
\draw (8,-1.5) node [scale=.5](123){};

\draw (101) to (111);
\draw (111) to (121);
\draw (121) [bend right=25, dashed] to (101);
\draw (100) to (112);
\draw (112) to (123);
\draw (123) [dashed] to (100);
\draw (102) to (113);
\draw (113) to (120);
\draw (120) [bend left=13, dashed] to (102);
\draw (103) to (110);
\draw (110) to (122);
\draw (122) [dashed] to (103);

\draw (0,-3) node [scale=.5] (200){};
\draw (1,-3) node [scale=.5] (210){};
\draw (2,-3) node [scale=.5] (220){};
\draw (0,-3.5) node [scale=.5](201){};
\draw (1,-3.5) node [scale=.5](211){};
\draw (2,-3.5) node [scale=.5](221){};
\draw (0,-4) node [scale=.5](202){};
\draw (1,-4) node [scale=.5](212){};
\draw (2,-4) node [scale=.5](222){};
\draw (0,-4.5) node [scale=.5](203){};
\draw (1,-4.5) node [scale=.5](213){};
\draw (2,-4.5) node [scale=.5](223){};

\draw (202) to (212);
\draw (212) to (222);
\draw (222) [bend right=25, dashed] to (202);
\draw (201) to (210);
\draw (210) to (223);
\draw (223) [bend left=15, dashed] to (201);
\draw (200) to (213);
\draw (213) to (221);
\draw (221) [dashed] to (200);
\draw (203) to (211);
\draw (211) to (220);
\draw (220) [bend left=25, dashed] to (203);

\draw (6,-3) node [scale=.5] (300){};
\draw (7,-3) node [scale=.5] (310){};
\draw (8,-3) node [scale=.5] (320){};
\draw (6,-3.5) node [scale=.5](301){};
\draw (7,-3.5) node [scale=.5](311){};
\draw (8,-3.5) node [scale=.5](321){};
\draw (6,-4) node [scale=.5](302){};
\draw (7,-4) node [scale=.5](312){};
\draw (8,-4) node [scale=.5](322){};
\draw (6,-4.5) node [scale=.5](303){};
\draw (7,-4.5) node [scale=.5](313){};
\draw (8,-4.5) node [scale=.5](323){};

\draw (303) to (313);
\draw (313) to (323);
\draw (323) [bend left=25, dashed] to (303);
\draw (301) to (312);
\draw (312) to (320);
\draw (320) [dashed] to (301);
\draw (302) to (310);
\draw (310) to (321);
\draw (321) [dashed] to (302);
\draw (300) to (311);
\draw (311) to (322);
\draw (322) [bend left=15, dashed] to (300);
\node[left=10pt,fill=none,draw=none] at (0,-.75) {$\Gamma(0)=$};
\node[left=10pt,fill=none,draw=none] at (6,-.75) {$\Gamma(1)=$};
\node[left=10pt,fill=none,draw=none] at (0,-3.75) {$\Gamma(2)=$};
\node[left=10pt,fill=none,draw=none] at (6,-3.75) {$\Gamma(3)=$};
\end{tikzpicture}
\end{center}
Note that the edges that join $G_0$ to $G_2$ are dashed since they will need to be distinguished from the other two edges in each $C_3$. 
It is easy to see that $\bigoplus_{i=0}^{3} \Gamma_{i}$ is a $C_{3}$-factorization of $K_{(4:3)}$.

\begin{lemma}
\label{K444}
There exist a decomposition of $K_{(4:3)}$ into $s$ $C_6$-factors and $4-s$ $C_3$-factors for any $s\in \{0,2,3,4\}$.
\end{lemma}
\begin{proof}
Consider the $C_3$ factorization of $K_{(4:3)}$, $\bigoplus_{i=0}^{3} \Gamma_{i}$. 
Let $\Lambda(\alpha,\beta)$ be the graph that has edges between $G_0$ (the first column) and $G_1$ (the second column) from $\Gamma(\alpha)$, has edges between $G_1$ and $G_2$ from $\Gamma(\alpha)$, and has dashed edges from $\Gamma(\beta)$. Notice that if $\alpha\neq \beta$ then $\Lambda(\alpha,\beta)$ is a union of cycles of size $6$.

This way we can get 2 $C_6$-factors by using $\Lambda(0,1)$ and $\Lambda(1,0)$ instead of $\Gamma(0)$ and $\Gamma(1)$ . We can get 3 $C_6$-factors by using edges $\Lambda(0,1)$, $\Lambda(1,2)$ and $\Lambda(2,1)$ instead of $\Gamma(0)$, $\Gamma(1)$ and $\Gamma(2)$. And finally we can get 4 $C_6$-factors by using $\Lambda(0,1)$, $\Lambda(1,2)$, $\Lambda(2,3)$ and $\Lambda(3,0)$.
This construction gives the desired decompositions.
\end{proof}

For $\bar{x}=1$, Lemma~{\normalfont\ref{K444}} gives a decomposition of $K_{(4\bar{x}:3)}$ into triangle factors and $C_{6\bar{x}}$-factors. We will extend this result to work on any $K_{(4\bar{x}:3)}$ where $\bar{x}>1$ and odd.
We are going to define two types of subgraphs, $T_{2\bar{x}}(\alpha,i)$ and $H_{2\bar{x}}(\alpha,i)(\beta,j)$ with a similar relation as the one between $\Gamma(\alpha)$ and $\Lambda(\alpha,\beta)$ (or $T_{x}(i)$ and $H_{x}(i,j)$ from Lemma {\normalfont\ref{Alg1}}).
Take $K_{(4:3)}$, and give weight $\bar{x}$ to each vertex. Now each triangle in $\bigoplus_{i=0}^3\Gamma_i$ becomes a copy of $K_{(\bar{x}:3)}$. Decompose these copies of $K_{(\bar{x}:3)}$ into triangles using Lemma {\normalfont\ref{Alg1}}. This gives a decomposition of $K_{(4\bar{x}:3)}$ into triangle factors.

Let $T_{2\bar{x}}(\alpha,i)$ be a triangle factor of $K_{(4\bar{x}:3)}$, where $0\leq \alpha \leq 3$ tells us from which $\Gamma(\alpha)$ it came and $0\leq i\leq \bar{x}-1$ tells us from which triangle factor $T_{\bar{x}}(i)$ of $K_{(\bar{x}:3)}$ it came.
Define $H_{2\bar{x}}(\alpha,i)(\beta,j)$ as the graph obtained by taking $T_{2\bar{x}}(\alpha,i)$ and replacing the edges between $G_0$ (the first column of $K_{(4\bar{x}:3)}$) and $G_2$ (the third column of $K_{(4\bar{x}:3)}$) with the same edges from $T_{2\bar{x}}(\beta,j)$. In this way we have that $H_{2\bar{x}}(\alpha,i)(\beta,j)\oplus H_{2\bar{x}}(\beta,j)(\alpha,i)=T_{2\bar{x}}(\alpha,i)\oplus T_{2\bar{x}}(\beta,j)$.

If $g\in H_{2\bar{x}}(\alpha,i)(\beta,j)$ is a vertex, we can think of it as a pair of coordinates $g=(g_1,g_2)$, with $g_1\in V(K_{(4:3)})$ and $g_2\in V(K_{(\bar{x}:3)})$. This is telling us from which vertex in $V(K_{(4:3)})$ and which vertex in $V(K_{(\bar{x}:3)})$ our vertex $g$ came.
Notice that when $\alpha\neq \beta$ the $K_{(4:3)}$ structure of $H_{2\bar{x}}(\alpha,i)(\beta,j)$ is a $C_{6}$-factor. This means that if we move through a cycle in $H_{2\bar{x}}(\alpha,i)(\beta,j)$ containing the vertex $(g_1,g_2)$, we will go through a vertex with first coordinate $g_1$ every six vertices.
In a similar fashion, when $\gcd (i-j,\bar{x})=1$ the $K_{(\bar{x}:3)}$ structure of the graph is a $C_{3\bar{x}}$-factor. This means that if we move through a cycle in $H_{2\bar{x}}(\alpha,i)(\beta,j)$ containing the vertex $(g_1,g_2)$, we will go through a vertex with second coordinate $g_2$ every $3\bar{x}$ vertices. Then if $\alpha\neq \beta$ and $\gcd (j-\beta,\bar{x})=1$, we are going to go through $(g_1,g_2)$ every $\lcm(6,3\bar{x})=6\bar{x}$ vertices. Hence $H_{2\bar{x}}(\alpha,i)(\beta,j)$ is a $C_{6\bar{x}}$-factor.

Let $\psi$ be a bijection on $\{(\alpha,i)|0\leq \alpha \leq 3, 0\leq i\leq \bar{x}-1\}$.
The previous discussion leads us to the following result.

\begin{lemma}\label{thislemma}
Let $\bar{x}$ be odd. Let $s$ and $r$ be non-negative integers such that $s+r=4\bar{x}$. If $\psi$ satisfies the following:
\begin{packed_item}
\item $\psi(\alpha,i)=(\alpha,i)$ for $r$ pairs $(\alpha,i)$; and 
\item $\psi(\alpha,i)=(\beta,j)$ with $\alpha\neq \beta$ and $\gcd (i-j,\bar{x})=1$ for the $s$ remaining pairs;
\end{packed_item}
then $K_{(4\bar{x}:3)}=\bigoplus H_{2\bar{x}}(\alpha,i)(\psi(\alpha,i))$ is a decomposition of $K_{(4\bar{x}:3)}$ into $r$ triangle factors and $s$ $C_{6\bar{x}}$-factors.
\end{lemma}

\begin{proof}
Notice that $H_{2\bar{x}}(\alpha,i)(\alpha,i)=T_{2\bar{x}}(\alpha,i)$, so if $\psi(\alpha,i)=(\alpha,i)$, $H_{2\bar{x}}(\alpha,i)(\psi(\alpha,i))$ is a triangle factor. When $\psi(\alpha,i)=(\beta,j)$ with $\alpha \neq \beta$ and $\gcd (i-j,\bar{x})=1$, by the discussion preceding the lemma, $H_{2\bar{x}}(\alpha,i)(\psi(\alpha,i))$ is a $C_{6\bar{x}}$-factor.
Therefore $K_{(4\bar{x}:3)}=\bigoplus H_{2\bar{x}}(\alpha,i)(\psi(\alpha,i))$ is a decomposition of $K_{(4\bar{x}:3)}$ into $r$ triangle factors and $s$ $C_{6\bar{x}}$-factors.
\end{proof}

Thanks to Lemma~{\normalfont\ref{thislemma}} we only need to show that for any $r\in \{0,1,\ldots,4\bar{x}-2,4\bar{x}\}$ we have a bijection $\psi$ satisfying the conditions of the lemma and with $r$ fixed points.

\begin{theorem}\label{Alg2}
Let $\bar{x}$ be odd and $s\in \{0,2,3,\ldots,4\bar{x}-1,4\bar{x}\}$, then:
\begin{packed_item}
\item There exists a bijection $\psi$ satisfying the conditions of Lemma {\normalfont\ref{thislemma}} with $r=4\bar{x}-s$ fixed points.
\item $K_{(4\bar{x}:3)}$ can be decomposed into $s$ $C_{6\bar{x}}$-factors and $r$ triangle factors. 
\end{packed_item}
\end{theorem}

\begin{proof}
If $s=0$ we just use the identity mapping.

If $2\leq s\leq 4\bar{x}$ we let $s_0$, $s_1$, $s_2$, $s_3\in\{0,2,3\ldots,\bar{x}-1\}$ be such that $s=s_0+s_1+s_2+s_3$.
We define $\psi$ as follows, where $m\in \{0,1,2,3\}$ and $i+m$ is taken$\pmod 4$:

\[
\psi(i+m,i)=\left\lbrace \begin{array}{lcl} 
(m,0) & \text{for}  & i=1\\
(i+m+2,i+2) & \text{for}  & i\equiv 0 \pmod 2, 0\leq i \leq s_m -3\\
(i+m-2,i-2) & \text{for}  & i\equiv 1 \pmod 2, 3\leq i \leq s_m -1\\
(s_m+m-2,s_m-2) & \text{for}  & i\equiv 0 \pmod 2, i=s_m-1\\
(s_m+m-1,s_m-1) & \text{for}  & i\equiv 0 \pmod 2, i=s_m-2\\
(i+m,i) &\text{for} & s_m\leq i \leq \bar{x}-1\\

\end{array}\right.
\]

It is an easy exercise to check that $\psi$ is a bijection with $4\bar{x}-(s_0+s_1+s_2+s_3)=r$ fixed points. Notice that $\psi(\alpha,i)-(\alpha,i)\in \{(0,0),(\pm 1,\pm 1),(\pm 2,\pm 2)\}$. This gives that if $\psi(\alpha,i)=(\beta,j)$ is not a fixed point of $\psi$, $\alpha\neq \beta$ and $\gcd(i-j,\bar{x})=1$.

Hence by Lemma~\ref{thislemma}
\[
K_{(4\bar{x}:3)}=\bigoplus H_{2\bar{x}}(\alpha,i)(\psi(\alpha,i))
\]
is a decomposition of $K_{(4\bar{x}:3)}$ into $s$ $C_{6\bar{x}}$-factors and $4\bar{x}-s$ triangle factors.
\end{proof}

\subsection{A Weighting Construction}
\label{A Weighting Construction}

A \emph{group divisible design} $(k,\lambda)\-\GDD(h^u)$ is a triple $(\mathcal{V},\mathcal{G},\mathcal{B})$ where $\mathcal{V}$ is a finite set of size $v=hu$, $\mathcal{G}$ is a partition of $\mathcal{V}$ into $u$ \emph{groups} each containing $h$ elements, and $\mathcal{B}$ is a collection of $k$ element subsets of $\mathcal{V}$ called \emph{blocks} which satisfy the following properties.
\begin{packed_item}
 \item If $B\in \mathcal{B}$, then $|B|=k$.
 \item If a pair of elements from $\mathcal{V}$ appear in the same group, then the pair cannot be in any block.
 \item Two points that are not in the same group, called a \emph{transverse pair}, appear in exactly $\lambda$ blocks.
 \item $|\mathcal{G}|>1$.
\end{packed_item}
These groups are not to be confused with the cyclic groups that were discussed earlier, which are algebraic groups.
A \emph{resolvable $\GDD$} (RGDD) has the additional condition that the blocks can be partitioned into parallel classes such that each element of $\mathcal{V}$ appears exactly once in each parallel class.
If $\lambda=1$, we refer to the RGDD as a $k$-RGDD$(h^u)$. In this paper, we will only talk about RGDDs
with $\lambda=1$. Necessary and sufficient conditions for the existence of $3$-RGDD$(h^u)$s have been established except in a finite number of cases.

\begin{theorem}\label{3rgdd}
 {\normalfont\cite{R}}  A $(3,\lambda)$-RGDD$(h^u)$ exists if and only if $u\geq 3,\lambda h(u-1)$ is even, $hu\equiv 0\pmod{3}$, and $(\lambda,h,u)\not\in\{(1,2,6),(1,6,3) \}\bigcup\{(2j+1,2,3), (4j+2,1,6):j\geq 0\}$.
\end{theorem}

In particular, we have that a $3$-RGDD$(3^u)$ exists for all odd $u\geq 3$ and a $3$-RGDD$(6^u)$ exists for all $u\geq 4$.

\begin{lemma}
\label{main}
Let $m \geq 3$, $n \geq 3$ and $x$ be positive integers such that both $m$ and $n$ divide $3x$. Suppose the following conditions are satisfied:
\begin{packed_item}
\item There exists a $3$-RGDD$(h^{u})$,
\item there exists a decomposition of $K_{(x:3)}$ into $r_{p}$ $C_{m}$-factors and $s_{p}$ $C_{n}$-factors, for\\
$p \in \{1,2, \ldots, \frac{h(u-1)}{2}\}$,
\item there exists an $(m,n)\-\HWP(hx;r_{\beta}, s_{\beta})$.
\end{packed_item}
Let
\[r_{\alpha}=\sum_{p=1}^{\frac{h(u-1)}{2}}r_{p} \mbox{ and } s_{\alpha}=\sum_{p=1}^{\frac{h(u-1)}{2}}s_{p}.\]
Then there exists a $(m,n)\-\HWP(hux;r_{\alpha}+r_{\beta},s_{\alpha}+s_{\beta})$.
\end{lemma}

\begin{proof}
Let $\{\p_1,\p_2, \ldots, \p_{\frac{h(u-1)}{2}}\}$ denote the parallel classes of the $3$-RGDD$(h^u)$,
and let $W=\{1,2, \ldots, x\}$. Consider each parallel class $\p_{p}$ with $p \in \{1,2, \ldots, \frac{h(u-1)}{2}\}$.
For each block $\{a_1, a_2, a_3\} \in \p_p$, construct a decomposition of $K_{(x:3)}$ into $r_p$
$C_m$-factors and $s_p$ $C_{n}$-factors with parts $\{a_i\} \times W$, for $i=1,2,3$.
Thus we have a decomposition of $K_{(hx:u)}$ into $r_{\alpha}$ $C_m$-factors and
$s_{\alpha}$ $C_{n}$-factors where

\[r_{\alpha}=\sum_{p=1}^{\frac{h(u-1)}{2}}r_{p} \mbox{ and } s_{\alpha}=\sum_{p=1}^{\frac{h(u-1)}{2}}s_{p}.\]

Now each part of $K_{(hx:u)}$ can be decomposed into $r_{\beta}$ $C_m$-factors
and $s_{\beta}$ $C_{n}$-factors. Thus there exists an $(m,n)\-\HWP(hux;r,s)$
where $r=r_{\alpha}+r_{\beta}$ and $s=s_{\alpha}+s_{\beta}$.
\end{proof}

\begin{lemma}
\label{main2}
Let $m\geq 3$, $n \geq 3$ and $x$ be positive integers such that both $m$ and $n$ divide $3x$. Suppose the following conditions
are satisfied:
\begin{packed_item}
\item There exists a $3\-$RGDD$(h^{u})$,
\item there exists an $(m,n)\-\HWP(3x;r_{\beta}, s_{\beta})$,
\item there exists a decomposition of $K_{(x:h)}$ into $r_{\gamma}$ $C_{m}$-factors and $s_{\gamma}$ $C_n$-factors,
\item there exists a decomposition of $K_{(x:3)}$ into $r_{p}$ $C_{m}$-factors and $s_{p}$ $C_{n}$-factors,
for \\$p \in \{1,2, \ldots, \frac{h(u-1)}{2}\}$.

\end{packed_item}
Let
\[r_{\alpha}=\sum_{p=1}^{\frac{h(u-1)}{2}-1}r_{p} \mbox{ and } s_{\alpha}=\sum_{p=1}^{\frac{h(u-1)}{2}-1}s_{p}.\]
Then there exists a $(m,n)\-\HWP(hux;r_{\alpha}+r_{\beta}+r_{\gamma},s_{\alpha}+s_{\beta}+s_{\gamma})$.
\end{lemma}

\begin{proof}
Let $\{\p_1,\p_2, \ldots, \p_{\frac{h(u-1)}{2}}\}$ denote the parallel classes of the $3$-RGDD$(h^u)$,
and let $W=\{1,2, \ldots, x\}$. Consider each parallel class $\p_p$ with $p \in \{1,2, \ldots, \frac{h(u-1)}{2}-1\}$.
For each block $\{a_1, a_2, a_3\} \in \p_p$, construct a decomposition of $K_{(x:3)}$ into $r_p$
$C_m$-factors and $s_p$ $C_{n}$-factors with parts $\{a_i\} \times W$, $i=1,2,3$.
For each block $\{a_1, a_2, a_3\}$ in parallel class $\p_{\beta}$ where $\beta=\frac{h(u-1)}{2}$, construct
an $(m,n)\-\HWP(3x;r_{\beta},s_{\beta})$ on $\{a_{1} \times W, a_{2} \times W, a_{3} \times W\}$.
Take a decomposition of $K_{(x:h)}$ into $r_{\gamma}$ $C_{m}$-factors and $s_{\gamma}$ $C_{n}$-factors
simultaneously on each group of the $3$-RGDD$(h^{u})$. 
This makes an $(m,n)\-\HWP(hux;r,s)$
where $r=r_{\alpha}+r_{\beta}+r_{\gamma}$ and $s=s_{\alpha}+s_{\beta}+s_{\gamma}$.
\end{proof}

\section{Main Results}
\label{general}

In this section, we use the constructions given in Section~{\normalfont\ref{Constructions}} to obtain results on the existence of a $(3,3x)\-\HWP(3xy;r,s)$.
We consider four different cases depending on the parity of $x$ and $y$.

\begin{lemma}
\label{double}
Suppose $x$ is even.  If there exists a decomposition of $K_{3x}-F$ into $r_{\delta}$ $C_{3}$-factors and $s_{\delta}$ Hamilton cycles, then there exists a decomposition of $K_{6x}-F$ into $r_{\delta}$ $C_{3}$-factors and $s_{\delta}+\frac{3x}{2}$ $C_{3x}$-factors.
\end{lemma}

\begin{proof}
Let $G_{1}$ and $G_{2}$ be a partition of the $6x$ points into two subsets of size $3x$. Decompose $G_{1}$ and $G_{2}$ into $r_{\delta}$ $C_{3}$-factors, $s_{\delta}$ Hamilton cycles, and a $1$-factor, $F$.  By Theorem~{\normalfont\ref{equip}}, there exists a decomposition of $K_{3x:2}$ into $\frac{3x}{2}$ $C_{3x}$-factors. The union of these edges is $K_{6x}$.
\end{proof}

\begin{theorem}
\label{x odd, y odd}
For each pair of odd integers $x \geq 3$ and $y \geq 3$, there exists a $(3,3x)\-\HWP(3xy; r,s)$ 
if and only if $r+s= \frac{v-1}{2}$ except when $s=1$ and $x=3$, and possibly when $s=1$
and 
\newline $x \in \{31, 37, 41, 43, 47, 51, 53, 59, 61, 67, 69, 71, 79, 83\}$.
\end{theorem}

\begin{proof}
By Theorem~{\normalfont\ref{3rgdd}} there exists a $3$-RGDD$(3^y)$ for all odd $y \geq 3$. There exists a decomposition of $K_{(x:3)}$ into $r_{p}$ $C_{3}$-factors and $s_{p}$ $C_{3x}$-factors for $(r_{p}, s_{p}) \in \{(x,0), (x-2,2), (x-3,3), \ldots, (0,x)\}$ by Theorem~{\normalfont\ref{equip}}. There exists a $(3,3x)\-\HWP(3x;r_{\beta}, s_{\beta})$ whenever $(r_{\beta}, s_{\beta}) \in \{(\frac{3x-1}{2}, 0), (\frac{3x-3}{2},1), (0, \frac{3x-1}{2})\}$ by Theorems~{\normalfont\ref{OP}} and ~{\normalfont\ref{Hammy}} (excluding the exception and possible exceptions listed in the statements of these theorems).  So apply Lemma~{\normalfont\ref{main}} with $m=3$ and $n=3x$.
We must now show that for each $s \in \{0,1, \ldots, \frac{3xy-1}{2}\}$, there exists a $(3,3x)\-\HWP(3xy;r,s)$. It is easy to see that if $s_{\alpha} \in \{0,2,3, \ldots, \frac{3xy-3x}{2}\}$, then we can write $s_{\alpha}=\sum_{i=1}^{(3y-3)/2} s_{p}$ where $s_{p} \in \{0,2,3, \ldots, x\}$. Thus if $s \in \{0,2,3, \ldots, \frac{3xy-3x}{2}\}$, then we may write $s=s_{\alpha}+s_{\beta}$ by choosing $s_{\alpha}=s$ and $s_{\beta}=0$. If $s=1$, then choose $s_{\alpha}=0$ and $s_{\beta}=1$. If $s=\frac{3xy-3x}{2}+1$, choose $s_{\alpha}=\frac{3xy-3x}{2}$ and $s_{\beta}=1$. Finally, let $i=2,3, \ldots, \frac{3x-1}{2}$, and consider $s=\frac{3xy-3x}{2}+i$. We may choose $s_{\alpha}=s-(\frac{3x-1}{2})$ and $s_{\beta}=\frac{3x-1}{2}$ because
\[2 \leq s-\frac{3x-1}{2} \leq \frac{3xy-3x}{2}.\]

\end{proof}

\begin{theorem}
\label{x odd, y even}
For each odd integer $x\geq 3$ and each even integer $y\geq 8$, there exists a 
\newline $(3,3x)\-\HWP(3xy;r,s)$ if and only if $r+s=\frac{3xy-1}{2}$ except possibly when $s=1$. 
\end{theorem}

\begin{proof}
By Theorem~{\normalfont\ref{3rgdd}}, there exists a $3$-RGDD$(6^{y/2})$ for all even $y\geq 8$. By Theorem~{\normalfont\ref{equip}}, for each $p\in\{1,2,\ldots,\frac{6(y/2-1)}{2}\}$, $K_{(x:3)}$ can be decomposed into $r_p$ $C_3$-factors and $s_p$ $C_{3x}$-factors where $(r_p,s_p)\in \{(x,0),(x-2,2),(x-3,3),\ldots,(0,x)\}$, so that $r_\alpha=\sum_{p=1}^{3(y/2-1)} r_p$ and $s_\alpha =\sum_{p=1}^{3(y/2-1)} s_p$.
By Theorem~{\normalfont\ref{OP}}, $K_{6x}$ can be decomposed into $r_\beta$ $C_3$-factors, $s_\beta$ $C_{3x}$-factors, and a $1$-factor where $(r_\beta,s_\beta)\in\{((6x-2)/2,0),(0,(6x-2)/2)\}$.
We must show that for each $s\in\{0,2,3,\ldots,(3xy-2)/2\}$ there exists a $(3,3x)\-\HWP(3xy;r,s)$. It is easy to see that such a decomposition exists when $s\in \{0,2,3,\ldots,(3xy-6x)/2\}$ by choosing $s_\alpha=s$ and $s_\beta=0$. For each $i\in\{1,2,\ldots,(6x-2)/2\}$, when $s=(3xy-6x)/2+i$, choose $s_\alpha=s-(6x-2)/2$ and $s_\beta=(6x-2)/2$. Notice that
$$2 \leq s_\alpha=\frac{3xy-6x}{2}+i-\left(\frac{6x-2}{2}\right) \leq \frac{3xy-6x}{2}+\left(\frac{6x-2}{2}\right)-\left(\frac{6x-2}{2}\right)\leq \frac{3xy-6x}{2}.$$
Therefore by Lemma~{\normalfont\ref{main}}, the proposed $(3,3x)\-\HWP(3xy;r,s)$ exists for all specified pairs $(r,s)$.
\end{proof}

\begin{theorem}
\label{x even, y odd}
For each even integer $x \geq 8$  and each odd integer $y\geq 3$, there exists a 
\newline $(3,3x)\-\HWP(3xy;r,s)$ if and only if $r+s=\frac{3xy-2}{2}$ except possibly when:

\begin{packed_item}
\item $(s,x) \in \{(2,12), (4,12)\},$
\item $1 \leq s \leq \frac{x}{2}-1$ and $x \equiv 4 \pmod{6}$, 
\item $s=1$ and $x \equiv 2 \pmod{12}$.
\end{packed_item}
 
\end{theorem}

\begin{proof}
Suppose $x \geq 8$ is even. By Theorem~{\normalfont\ref{3rgdd}}, there exists a $3$-RGDD$(3^{y})$ for all odd integers $y \geq 3$. By Theorem~{\normalfont\ref{equip}}, for each $p \in \{1,2, \ldots,
\frac{3(y-1)}{2}\}$, $K_{(x:3)}$ can be decomposed into $r_{p}$ $C_{3}$-factors and $s_{p}$ $C_{3x}$-factors, where $(r_{p},s_{p}) \in
\{(x,0),(0,x)\}$. By Theorem~{\normalfont\ref{Hammy}}, there exists a decomposition of $K_{3x}$ into $r_{\beta}$ $C_{3}$-factors and $s_{\beta}$ $C_{3x}$-factors and a 1-factor for $(r_{\beta},s_{\beta}) \in \{(\frac{3x-2}{2},0), (\frac{3x-4}{2},1),\ldots, (0,\frac{3x-2}{2})\}$, except possibly when 
$(s_{\beta},x) \in \{(2,12),(4,12)\}$; $1 \leq s_{\beta} \leq \frac{x}{2}-1$ and $x \equiv 4 \pmod{6}$; or $s_{\beta}=1$ and $x \equiv 2 \pmod{12}$.
We apply Lemma~{\normalfont\ref{main}} to obtain a $(3,3x)\-\HWP(3xy;r,s)$ with
$r=r_{\alpha}+r_{\beta}$ and $s=s_{\alpha}+s_{\beta}$ for all
$s \in \{0,1, \ldots, \frac{3xy-2}{2}\}$ (with the exceptions listed in the statement of this theorem) as follows.
We may write  $s_\alpha =\sum_{p=1}^{\frac{3(y-1)}{2}} s_p$ where $s_{p}\in \{0,x\}$,
so that $s_{\alpha} \in \{0,x,2x, \ldots, x \cdot \frac{3y-3}{2}\}$.
Write $s=t \cdot x+i$, where $t \in \{0,1, \ldots, \frac{3y-3}{2}\}$
and $ i \in \{0, 1, \ldots, \frac{3x-2}{2}\}$. We may choose $s_{\alpha}=s-i$
and $s_{\beta}=i$.

\end{proof}

Note that the cases of $x=2,4$ are not considered in the previous theorem. They will be handled in Section~{\normalfont\ref{When $x$ is small}}.
We leave open the case of $x=6$ and $y$ odd.

\begin{theorem}
\label{x even, y even}
For each even integer $x \geq 8$ and each even integer $y\geq 8$, there exists a 
\newline $(3,3x)\-\HWP(3xy;r,s)$ if and only if $r+s=\frac{3xy-2}{2}$ except possibly when:

\begin{packed_item}
\item $(s,x) \in \{(2,12), (4,12)\}$,
\item $2 \leq s \leq \frac{x}{2}-1$ and $x \equiv 4$ or $10 \pmod{12}$,
\item $s=1$ and $x \equiv 2,4,$ or $10 \pmod{12}$.
\end{packed_item}
 
\end{theorem}

\begin{proof}
There exists a $3$-RGDD$(6^{y/2})$ for all even $y \geq 8$ by Theorem~{\normalfont\ref{3rgdd}}.  
There exists a decomposition of $K_{(x:3)}$ into $r_{p}$ $C_{3}$-factors
and $s_{p}$ $C_{3x}$-factors for $(r_{p},s_{p}) \in \{(0,x),(x,0)\}$ by Theorem~{\normalfont\ref{equip}}.
By the same result, we also get a decomposition of $K_{(x:6)}$ into $r_{\gamma}$ $C_{3}$-factors
and $s_{\gamma}$ $C_{3x}$-factors for $(r_{\gamma}, s_{\gamma}) \in \{(0,\frac{5x}{2}), (\frac{5x}{2},0)\}$. By Theorem~{\normalfont\ref{Hammy}}, there exists a decomposition of $K_{3x}$ into $r_{\beta}$ $C_{3}$-factors, $s_{\beta}$ $C_{3x}$-factors, and a 1-factor for $(r_{\beta}, s_{\beta}) \in \{\frac{3x-2}{2}, 0), (\frac{3x-4}{2}, 1),\ldots, (0, \frac{3x-2}{2})\}$, except possibly when
$(s_{\beta},x) \in \{(2,12), (4,12)\}$;
$1 \leq s_{\beta} \leq \frac{x}{2}-1$ and $x \equiv 4 \pmod{6}$; or
$s_{\beta}=1$ and $x \equiv 2 \pmod{12}$.
Write $s_{\alpha}=\sum_{p=1}^{\frac{3y}{2}-4} s_{p}$ so $s_{\alpha} \in \{0,x,2x,\ldots, x(\frac{3y}{2}-4)\}$.
By Lemma~{\normalfont\ref{main2}}, we obtain a $(3,12)\-\HWP(3xy;r,s)$ for all $s \in \{0,1,\ldots, \frac{3xy-2}{2}\}$
as follows.
If $s \in \{0,1, \ldots, \frac{3xy}{2}-\frac{5x}{2}-1\}$, it is easy to see that we can let $s_{\gamma}=0$
and write $s$ as $s=s_{\alpha}+s_{\beta}$. 
If $s=\frac{3xy}{2}-\frac{5x}{2}+i$, for $i=0,1,\ldots, \frac{3x}{2}-1$ choose $s_{\alpha}=(\frac{3y}{2}-5)x$,
$s_{\beta}=i$, and $s_{\gamma}=\frac{5x}{2}$.
If $s=\frac{3xy}{2}-x+i$ for $i=0,1,\ldots, x-1$, choose $s_{\alpha}=(\frac{3y}{2}-4)x$, $s_{\beta}=\frac{x}{2}+i$ and $s_{\gamma}=\frac{5x}{2}$. 
\end{proof}

We can fill in some of the gaps that we have left by using Theorem {\normalfont\ref{Alg2}}.
\begin{theorem}\label{barxtheo}
For each odd integer $\bar{x}\geq 3$ and each even integer $y\geq 6$, there exists a \newline $(3,6\bar{x})\-\HWP(6\bar{x}y;r,s)$ if and only if $r+s=\frac{6\bar{x}y-2}{2}$ except possibly when $s=1$.
\end{theorem}

\begin{proof}
Assume that $y \equiv 2 \pmod{4}$ and $y \geq 6$. For all such $y$, 
there exists a $3$-RGDD$(3^{\frac{y}{2}})$  by Theorem~{\normalfont\ref{3rgdd}}.
There exists a $(3,6\bar{x})\-\HWP(12\bar{x};r_{\beta},s_{\beta})$ for all $(r_{\beta},s_{\beta}) \in \{(0,\frac{12\bar{x}-2}{2}),(\frac{12\bar{x}-2}{2},0)\}$ by
Theorem~{\normalfont\ref{OP}}. By Theorem~{\normalfont\ref{Alg2}}, we have that $K_{(4\bar{x}:3)}$ can be decomposed into $r_{p}$ $C_{3}$-factors
and $s_{p}$ $C_{6\bar{x}}$-factors for $(r_{p},s_{p}) \in \{(0,4\bar{x}),(1,4\bar{x}-1),\ldots,(4\bar{x}-2,2),(4\bar{x},0)\}$. 
Apply Lemma~{\normalfont\ref{main}} with $m=3$, $n=6\bar{x}$, and $x=4\bar{x}$. Let $s_{\alpha}=\sum_{p=1}^{3(\frac{y}{2}-1)/2} s_{p}$,
then it is easy to see that $s_{\alpha} \in \{0,2,3, \ldots, 3\bar{x}y-6\bar{x}\}$.
Write $s=s_{\alpha}+s_{\beta}$  where $s_{\alpha} \in \{0,2,3, \ldots, 3\bar{x}y-6\bar{x}\}$
and $s_{\beta} \in \{0,6\bar{x}-1\}$. Then we can write $s$ as $s_{\alpha}+s_{\beta}$
for every $s \in \{0,2,3, \ldots, \frac{6\bar{x}y-2}{2}\}$ in this way. Thus we can construct a $(3,6\bar{x})\-\HWP(6\bar{x}y;r,s)$ for all $s \in \{0,1, \ldots,
\frac{6\bar{x}y-2}{2}\}$.

Assume $y \equiv 0 \pmod{4}$, and $y \geq 12$.  Then there exists a $3$-RGDD$(6^{\frac{y}{4}})$ by Theorem~{\normalfont\ref{3rgdd}}.
There exists a decomposition of $K_{(4\bar{x}:3)}$ into $r_{p}$ $C_{3}$-factors and $s_{p}$ $C_{6\bar{x}}$-factors
for $s_{p} \in \{0,2,3,\ldots,4\bar{x}\}$ by Theorem~{\normalfont\ref{Alg2}}. By Theorem~{\normalfont\ref{equip}}, there exists a $(C_{3},C_{6\bar{x}})$-factorization
of $K_{(4\bar{x}:6)}$ for $(r_{\gamma},s_{\gamma}) \in \{(0,10\bar{x}),(10\bar{x},0)\}$. There exists a $(3,6\bar{x})\-\HWP(12\bar{x};r_{\beta},s_{\beta})$
for $s_{\beta} \in \{0,\frac{12\bar{x}-2}2\}$ by Theorem~{\normalfont\ref{OP}}.  Now we can easily write $s=s_{\alpha}+s_{\beta}+s_{\gamma}$
for $s \in \{0,2,3, \ldots, 3\bar{x}y-1\}$ and apply Lemma~{\normalfont\ref{main2}}.

\end{proof}

By writing $x=2\bar{x}$ Theorem~{\normalfont\ref{barxtheo}} covers the cases when $s \not=1$ and $x=6$ and also some of the 
cases when $s \not=1$ and $x\equiv 4 \pmod {6}$  (namely the ones where $x\equiv 10 \pmod{12}$). 
When $x \geq 6$ is even and $y \geq 8$ is even, the cases that are not covered by Theorems~{\normalfont\ref{x even, y even}} and ~{\normalfont\ref{barxtheo}}
are as follows:

\begin{packed_item}
\item $(s,x) \in \{(2,12), (4,12)\}$,
\item $2 \leq s \leq \frac{x}{2}-1$ and $x \equiv 4 \pmod{12}$,
\item $s=1$ and $x \equiv 2,4,10 \pmod{12}$.
\end{packed_item}

Because there is no $3$-RGDD$(6^u)$ for $u\leq 3$, Lemmas~{\normalfont\ref{main}} and {\normalfont\ref{main2}} are not useful when $y\in \{2,4,6\}$. However,
we still have some results.  When $y=2$ and $x$ is even we may apply Lemma~{\normalfont\ref{double}}
to find a $(3,3x)\-\HWP(6x;r,s)$ for $s=s_1+\frac{3x}{2},$ $r=r_1$, where $(s_1,r_1)$ is a solution of the Hamilton-Waterloo 
Problem with triangles and Hamilton cycles for $K_{3x}$.

When $y=4$ and $x \geq 2$ is even, consider $K_{12x}$. We can partition the vertices into four parts of size $3x$. In the four copies of $K_{3x}$ we have some solutions for the Hamilton-Waterloo Problem with triangles and Hamilton cycles. The remaining edges give us $K_{(3x:4)}$, which can be decomposed into all $C_{3x}$-factors or into all triangle factors. In this way we can get either all triangle factors, or $s=s_1+e_1\frac{9x}{2},$ $r=r_1+e_2\frac{9x}{2}$, where $(s_1,r_1)$ is a solution of the Hamilton-Waterloo problem with triangles and Hamilton cycles for $K_{3x}$ and $e_1+e_2=1$, $e_1,e_2\geq 0$.
If $y=6$ and $x$ is even, consider $K_{18x}$. 
By following the same method, 
we can get either all triangle factors, or $s=s_1+e_1\frac{15x}{2},$ $r=r_1+e_2\frac{15x}{2}$, where $(s_1,r_1)$ is a solution of the Hamilton-Waterloo Problem with triangles and Hamilton cycles for $K_{3x}$ and $e_1+e_2=1$, $e_1,e_2\geq 0$.
\subsection{When $x$ is small}
\label{When $x$ is small}

In this subsection, we consider the small values of $x$ for which the general constructions used in Section~{\normalfont\ref{general}} cannot be readily applied.
By applying the methods described at the end of Section~{\normalfont\ref{general}}, it is easy to see that the following decompositions exist
when $x=2$:
a $(3,6)\-\HWP(24;r,s)$ for $s \in \{0,1,2,7,8,9,10,11\}$, and a $(3,6)\-\HWP(48;r,s)$ for $s \in \{0,1,2,3,4,5,12,13,14,19,20,21,22,23\}$.
The following three results gives solutions to the Hamilton-Waterloo Problem, $(3,3x)\-\HWP(3xy;r,s)$, for all other values of $y$ when
$x=2$.

\begin{theorem}
\label{x=2,y=2mod4}
There exists a $(3,6)\-\HWP(6y;r,s)$ for all $y\equiv 2 \pmod{4}$ if and only if $r+s=\frac{6y-2}{2}$, except when $y=2$ and $s=0$. 
\end{theorem}

\begin{proof}
If $y=2$, then there exists a $(3,6)\-\HWP(12;r,s)$ for all possible $r$ and $s$ except when $s=0$ by Theorem~{\normalfont\ref{6and12}}.
We now assume that $y \equiv 2 \pmod{4}$ and $y \geq 6$. For all such $y$, 
there exists a $3$-RGDD$(3^{\frac{y}{2}})$  by Theorem~{\normalfont\ref{3rgdd}}.
There exists a $(3,6)\-\HWP(12;r_{\beta},s_{\beta})$ for all $(r_{\beta},s_{\beta}) \in \{(0,5),(1,4),(2,3),(3,2),(4,1)\}$ by
Theorem~{\normalfont\ref{6and12}}. By Lemma~{\normalfont\ref{K444}}, we have that $K_{(4:3)}$ can be decomposed into $r_{p}$ $C_{3}$-factors
and $s_{p}$ $C_{6}$-factors for $(r_{p},s_{p}) \in \{(0,4),(1,3),(2,2),(4,0)\}$. 
Apply Lemma~{\normalfont\ref{main}} with $m=3$, $n=6$, and $x=4$. Let $s_{\alpha}=\sum_{p=1}^{3(\frac{y}{2}-1)/2} s_{p}$,
then it is easy to see that $s_{\alpha} \in \{0,2,3, \ldots, 3y-6\}$.
Write $s=s_{\alpha}+s_{\beta}$  where $s_{\alpha} \in \{0,2,3, \ldots, 3y-6\}$
and $s_{\beta} \in \{1,2,3,4,5\}$. Then we can write $s$ as $s_{\alpha}+s_{\beta}$
for every $s \in \{1,2, \ldots, \frac{6y-2}{2}\}$ in this way. 
If $s=0$, then there exists a $(3,6)\-\HWP(6y;r,s)$
by Theorem~{\normalfont\ref{OP}}. Thus we can construct a $(3,6)\-\HWP(6y;r,s)$ for all $s \in \{0,1, \ldots,
\frac{6y-2}{2}\}$.

\end{proof}

\begin{theorem}
\label{x=2,y=0mod4}
There exists a $(3,6)\-\HWP(6y;r,s)$ for all $y\equiv 0 \pmod{4}$ if and only if $r+s=\frac{6y-2}{2}$, except possibly when
$y=4$ or $y=8$.
\end{theorem}

\begin{proof}
Assume $y \equiv 0 \pmod{4}$, and $y \geq 12$.  Then there exists a $3$-RGDD$(6^{\frac{y}{4}})$ by Theorem~{\normalfont\ref{3rgdd}}.
There exists a decomposition of $K_{(4:3)}$ into $r_{p}$ $C_{3}$-factors and $s_{p}$ $C_{6}$-factors
for $s_{p} \in \{0,2,3,4\}$ by Lemma~{\normalfont\ref{K444}}. By Theorem~{\normalfont\ref{equip}}, there exists a $(C_{3},C_{6})$-factorization
of $K_{(4:6)}$ for $(r_{\gamma},s_{\gamma}) \in \{(0,10),(10,0)\}$. There exists a $(3,6)\-\HWP(12;r_{\beta},s_{\beta})$
for $s_{\beta} \in \{1,2,3,4,5\}$ by Theorem~{\normalfont\ref{6and12}}.  Now we can easily write $s=s_{\alpha}+s_{\beta}+s_{\gamma}$
for $s \in \{0,1,\ldots, 3y-1\}$ and apply Lemma~{\normalfont\ref{main2}}.
\end{proof}

\begin{theorem}
\label{x=2, y odd}
There exists a $(3,6)\-\HWP(6y;r,s)$ when $y$ is odd and \\$s \in \{1,2, \frac{3(y-1)}{2}+1, \frac{3(y-1)}{2}+2,\ldots, 3y-1\}$.
\end{theorem}

\begin{proof}
If $y=1$, then exists a $(3,6)\-\HWP(6;r,s)$
for all possible $r$ and $s$ except for $(r,s)=(2,0)$ by Theorem~{\normalfont\ref{6and12}}.
Assume $y\geq 3$ is odd, then there exists a $3$-RGDD$(3^y)$ by Theorem~{\normalfont\ref{3rgdd}}.  There exists a $(3,6)\-\HWP(6;r_{\beta},s_{\beta})$
for $(r_{\beta},s_{\beta}) \in \{(1,1),(0,2)\}$ by Theorem~{\normalfont\ref{6and12}}. It is easy to see that $K_{(2:3)}$ can be decomposed into a $C_{3}$-factor and a $C_{6}$-factor or two $C_{6}$-factors.  Apply Lemma~{\normalfont\ref{main}} with $m=3$, $n=6$ and $x=2$. Let $s_{\alpha}=\sum_{p=1}^{3(y-1)/2} s_{p}$ with $s_{p} \in \{1,2\}$ and
notice that $s_{\alpha} \in \{\frac{3(y-1)}{2},\frac{3(y-1)}{2}+1, \ldots, 3(y-1)\}$. Then we can write $s$ as $s_{\alpha}+s_{\beta}$ for
every $s \in \{\frac{3(y-1)}{2}+1, \frac{3(y-1)}{2}+2,\ldots, 3y-1\}$. Thus we obtain a $(3,6)\-\HWP(6y;r,s)$ for all such $s$. We can also obtain a $(3,6)\-\HWP(6y;r,s)$
for $s=1$ and $s=2$ as follows. There exists a $3$-RGDD$(6^y)$ by Theorem~{\normalfont\ref{3rgdd}}; it has $3(y-1)$ parallel classes. There exists a $(3,6)\-\HWP(6;r_{\beta},s_{\beta})$ for $s_{\beta} \in \{1,2\}$. Apply Lemma~{\normalfont\ref{main}} with $m=3$, $n=6$ and $x=1$, and write $s=s_{\alpha}+s_{\beta}$ with $s_{\alpha}=0$ and $s_{\beta}=1$ or $s_{\beta}=2$.
\end{proof}

Recall from Theorem~{\normalfont\ref{6and12}} that there exists a $(3,12)\-\HWP(12;r_{\delta},s_{\delta})$ if and only if $s_{\delta} \in \{1,2,3,4,5\}$. 
For each possible decomposition of $K_{12}$, let $s_{\beta}=s_{\delta}+6$, and apply Lemma~{\normalfont\ref{double}}
to obtain a $(3,12)\-\HWP(24;r, s)$ for all $s \in \{7,8,9,10,11\}$. If $s=0$, then simply apply Theorem~{\normalfont\ref{OP}}.
Similarly, apply Theorem~{\normalfont\ref{OP}} to obtain a $(3,12)\-\HWP(48;r, s)$ for $s=0$.
Consider the equipartite graph $K_{(12:4)}$. It has a $C_{12}$-factorization
and a $C_{3}$-factorization by Theorem~{\normalfont\ref{equip}}. On each part, construct a $(3,12)\-\HWP(12;r,s)$ for
$s \in \{1,2,3,4,5\}$. Thus we have a $(3,12)\-\HWP(48;r,s)$ for $s \in \{0,1,2,3,4,5,19,20,21,22,23\}$.
The next theorem settles the Hamilton-Waterloo Problem, $(3,3x)\-\HWP(3xy;r,s)$ when
$x=4$ for the remaining values of $y$.

\begin{theorem}
\label{x=4}
For $y=3$ and all $y \geq 5$, there exists a $(3,12)\-\HWP(12y;r,s)$ if and only if $r+s=\frac{v-2}{2}$.
\end{theorem}

\begin{proof}
Let $y \geq 6$ be even. There exists a $3$-RGDD$(6^{y/2})$  by Theorem~{\normalfont\ref{3rgdd}}.
There exists a decomposition of $K_{(4:3)}$ into $r_{p}$ $C_{3}$-factors
and $s_{p}$ $C_{12}$-factors for $(r_{p},s_{p}) \in \{(0,4),(4,0)\}$ by Lemma~{\normalfont\ref{equip}}.
By the same result, we also get a decomposition of $K_{(4:6)}$ into $r_{\gamma}$ $C_{3}$-factors
and $s_{\gamma}$ $C_{12}$-factors for $(r_{\gamma}, s_{\gamma}) \in \{(0,10), (10,0)\}$.
Recall that there exists a $(3,12)\-\HWP(12;r_{\beta},s_{\beta})$ for $(r_{\beta},s_{\beta}) \in \{(0,5),(1,4),(2,3),(3,2),(4,1)\}$
by Theorem~{\normalfont\ref{6and12}}.
Write $s_{\alpha}=\sum_{p=1}^{\frac{3y}{2}-4} s_{p}$ so $s_{\alpha} \in \{0,4,8,\ldots, 6y-16\}$.
By Lemma~{\normalfont\ref{main2}}, we obtain a $(3,12)\-\HWP(3xy;r,s)$ for all $s \in \{0,1,\ldots, 6y-1\}$
as follows.
If $s=0$, apply Theorem~{\normalfont\ref{OP}}.
If $s \in \{1,2, \ldots, 6y-11\}$, it is easy to see that we can let $s_{\gamma}=0$
and write $s$ as $s=s_{\alpha}+s_{\beta}$. 
If $s=6y-10$, choose $s_{\alpha}=6y-24$, $s_{\beta}=4$, and $s_{\gamma}=10$.
If $s=6y-i$ for $i=9,8,7,6$, choose $s_{\alpha}=6y-20$, $s_{\beta}=10-i$
and $s_{\gamma}=10$. 
If $s=6y-i$ for $i=5,4,3,2,1$, choose $s_{\alpha}=6y-16$, $s_{\beta}=6-i$ and $s_{\gamma}=10$.

If $y \geq 3$ is odd, there exists a $3$-RGDD$(3^{y})$ by Lemma~{\normalfont\ref{3rgdd}}. There exists a decomposition of $K_{(4:3)}$
into $r_{p}$ $C_{3}$-factors and $s_{p}$ $C_{12}$-factors for $(r_{p},s_{p}) \in \{(0,4),(4,0)\}$ by Theorem~{\normalfont\ref{equip}}.
Write $s_{\alpha}=\sum_{p=1}^{\frac{3(y-1)}{2}} s_{p}$, so $s_{\alpha} \in \{0,4,8,\ldots, 6(y-1)\}$.
Recall the existence of a $(3,12)\-\HWP(12; r_{\beta},s_{\beta})$ for $s_{\beta} \in \{1,2,3,4,5\}$.
Then it is easy to see that we can write $s$ as $s_{\alpha}+s_{\beta}$ for all $s \in\{0,1,2, \ldots, 6y-1\}$.
Thus we may apply Lemma~{\normalfont\ref{main}} for the result.

\end{proof}

\section{Conclusions}
The following Theorem combines the results from Theorems~{\normalfont\ref{x odd, y odd}},~{\normalfont\ref{x odd, y even}},~{\normalfont\ref{x even, y odd}},~{\normalfont\ref{x even, y even}}, {\normalfont\ref{barxtheo}},~{\normalfont\ref{x=2,y=2mod4}},~{\normalfont\ref{x=2,y=0mod4}},~{\normalfont\ref{x=2, y odd}}, and ~{\normalfont\ref{x=4}} (note that we did not include all of the small partially complete results such as those at the end of Section~{\normalfont\ref{general}}).:
\begin{theorem}
Let $x\geq 2$, $y\geq 2$, and $r,s\geq 0$ such that $r+s=\lfloor\frac{3xy-1}{2}\rfloor$. Then there exist a $(3,3x)$-HWP$(3xy;r,s)$ except possibly when:
\begin{packed_item}
\item $s=1$, $y\geq 3$, and $x\in\{3,31,37,41,43,47,51,53,59,61,67,69,71,79,83\}$.
\item $s=1$, $x$ is odd and $y$ is even.
\item $s=1$, $x\geq 6$, $x\equiv 2 \pmod {12}$.
\item $s=1$, $y\geq 8$ is even and $x\equiv 10 \pmod{12}$.
\item $s=1$, $x\geq 3$ is odd and $y$ is even.
\item $1\leq s \leq \frac{x}{2}-1$, $x\geq 16$, $x\equiv 4 \pmod{12}$, $y$ is even.
\item $1\leq s \leq \frac{x}{2}-1$, $x\geq 10$, $x\equiv 4 \pmod{6}$, $y$ is odd.
\item $(s,x)\in \{(2,12),(4,12)\}$.
\item $s=0$, $x=2$, $y=2$.
\item $x=2$ and $y\in\{4,8\}$.
\item $s\in \{3,4,\ldots \frac{3(y-1)}{2}\}$, $x=2$ and $y\geq 3$ is odd.
\item $x\not\in\{2,4\}$ and $y\in\{2,4,6\}$.
\item $x=4$ and $y\in\{2,4\}$.
\item $x=6$ and $y$ odd.
\end{packed_item}

\end{theorem}

\label{conclusions}

\end{document}